\documentclass{amsart}

\usepackage{amssymb, latexsym, verbatim, color} 

\usepackage{graphicx, tikz, tikz-qtree}

\usepackage{algpseudocode, algorithm}

\usepackage{enumitem}

\algnewcommand{\Inputs}[1]{
\hspace*{\algorithmicindent}\parbox[t]{.8\linewidth}{\raggedright #1} }

\algnewcommand{\Initialize}[1]{
\hspace*{\algorithmicindent}\parbox[t]{.8\linewidth}{\raggedright #1} }

\makeatletter \def\BState{\State\hskip-\ALG@thistlm} \makeatother

\newtheorem*{theorem*}{Theorem} 

\newtheorem{theorem}{Theorem}

\newtheorem{lemma}{Lemma} 

\newtheorem{remark}{Remark}

\begin{document}

\title[Quasi-positivity in free groups]{Quasi-positivity and recognition
of products of conjugacy classes in free groups} 

\author{Robert W. Bell}

\address{Lyman Briggs College \& Department of Mathematics, Michigan
State University, East Lansing, MI} 

\email{rbell@math.msu.edu}

\author{Rita Gitik} 

\address{Department of Mathematics, University of
Michigan, Ann Arbor, MI} 

\email{ritagtk@umich.edu}

\date{29 January 2019}

\keywords{free group, quasi-positive, product of conjugates, complexity
of algorithms}

\subjclass{MSC 2010 20E05, 20E45, 20F10, 68Q25}

\begin{abstract} 

Given a group $G$ and a subset $X \subset G$, an element $g \in G$ is
called quasi-positive if it is equal to a product of conjugates of
elements in the semigroup generated by $X$.  This notion is important in
the context of braid groups, where it has been shown that the closure of
quasi-positive braids coincides with the geometrically defined class of
$\mathbb{C}$-transverse links.  We describe an algorithm that recognizes
whether or not an element of a free group is quasi-positive with respect
to a basis.  Spherical cancellation diagrams over free groups are used
to establish the validity of the algorithm and to determine the
worst-case runtime.  

\end{abstract}

\maketitle

\section{Introduction} 

A natural topological question is whether a given
link $L$ in 3-dimensional Euclidean space is equivalent to the closure
of a braid in which all of the strands cross in a positive sense.  Such
a braid is called a positive braid.  Analogously, one might study when a
link is equivalent to the closure of a product of conjugacy classes of
positive braids; such a link is called quasi-positive (see, for
instance, \cite{Abe_Tagami_MR3715009},
\cite{Etnyre_VanHornMorris_MR3609913},
\cite{Feller_Krcatovich_MR3694648}, \cite{Hayden_MR3788795},
\cite{Orevkov_MR3283750}, \cite{Orevkov_MR3444041},
\cite{Orevkov_MR3442790}, \cite{Orevkov_MR3420365},
\cite{Rudolph_MR683760}, \cite{Rudolph_MR1167166}, and
\cite{Stoimenow_MR3729301}).

Quasi-positivity has a striking geometric interpretation: Rudolph proved
that $\mathbb{C}$-transverse links, i.e. links which are the
intersection of a complex curve and a sphere in $\mathbb{C}^2$
transverse to the given curve, are quasi-positive;  and Boileau and
Orevkov proved the converse.  (See section 3.2.2 of
\cite{Etnyre_VanHornMorris_MR3609913} for references and further
discussion.)   

From an algebraic perspective, an element of the braid group is
quasi-positive if and only if it belongs to the semigroup that is
normally generated by the standard braid generators.

\subsection{Quasi-positivity in abstract groups.}

Quasi-positivity is a condition on elements of a group $G$ with a
distinguished subset $P$ consisting of a priori positive elements.  An
element $g \in G$ is defined to be quasi-positive if it is equal in $G$
to a product of conjugates of elements of $P$.  For our purposes, it is
convenient to regard the identity element as a quasi-positive element:
thus we call an element quasi-positive if it belongs to the normal
closure of the monoid generated by $P$.  

It seems natural to investigate quasi-positivity when $G = F(X)$ is a 
free group and $P = X$ is a basis.  Hence, we ask the following: 

\begin{quote} Can one decide whether or not a word over an alphabet $X$ 
represents a quasi-positive element of the free group $F(X)$?  
\end{quote}

Orevkov provided a positive answer to this question in his work on
quasi-positivity in the three strand braid
group~\cite{Orevkov_MR2056762}.  Orevkov also attributes a solution to 
this problem to Po\'{e}naru, as communicated by Blank
in~\cite{Poenaru_MR1610469}. 

In this article, we give a different algorithmic solution to this
question.  Our methods are topological and both extend and complement
the work of Orevkov.  We compare our methods to Orevkov's in
Section~\ref{S:OrevkovMethod}.  

We hope that this investigation may be useful in the study of
quasi-positivity in braid groups.  An immediate application is the
following: if $g \in F(x_1, \dots, x_{n-1})$ is quasi-positive, then $g$
maps to a quasi-positive element of the $n$-strand braid group under the
natural map that sends $x_i$ to the element $\sigma_i$ that braids the
$(i+1)$-st strand over the $i$-th strand.

\subsection{Products of conjugacy classes in free groups.}

Gersten~\cite{Gersten_MR840821} studied the following problem:  

\begin{quote} Given a sequence $w_1, \dots, w_k \in F(X)$, when do there
exist $c_1, \dots, c_k \in F(X)$ such that $w_1^{c_1} \cdots w_k^{c_k} =
1$?  \end{quote}

Gersten reformulates this question in terms of the topology of surfaces
as follows.  Since the sequence $w_1, \dots, w_k$ involves only
finitely many letters, we may suppose that $X$ is finite, say $X =
\{x_1, \dots, x_n\}$.  Let $M$ be an oriented $2$-sphere with $k$
mutually disjoint open disks removed.  Let $\Gamma_1, \dots, \Gamma_k$
be the oriented boundary circles of $M$.  For each $i = 1, \dots, k$,
subdivide $\Gamma_i$ by choosing $|w_i|$ points as vertices, where $|a|$
denotes the number of letters in the word $a$.  Choose one vertex to be
the initial vertex on $\Gamma_i$, and then label the initial vertex by
the first letter of $w_i$, the next vertex (in the order determined by
the orientation) by the next letter of $w_i$, and so forth.  For
instance, if $w_1 = x_2 x_1^{-1} x_2^{-1}x_1$, then $\Gamma_1$ is 
subdivided into four arcs by choosing four vertices with labels $x_2$, 
$x_1^{-1}$, $x_2^{-1}$, and $x_1$, respectively.
Gersten's theorem can now be stated as follows:

\begin{theorem*}[Gersten, Theorem 5.2 in~\cite{Gersten_MR840821}]
There exist $c_1, \dots, c_k \in F(X)$ such that $w_1^{c_1} \cdots
w_k^{c_k} = 1$ if and only if there exist pairwise disjoint properly
embedded tame arcs in $M$ which join vertices labeled by inverse
letters, i.e. one labeled by $x$ and one by $x^{-1}$ for some $x \in X$,
and such that every vertex is joined by an arc to exactly one other
vertex.  \end{theorem*}

A spherical diagram of the type described in the theorem is called a
spherical cancellation diagram for $w_1, \dots, w_k$ over $F(X)$.
Gersten outlines a proof of this theorem and states that this theorem
also follows from work of Culler~\cite{Culler_MR605653} and
Goldstein--Turner~\cite{Goldstein_Turner_MR521516}.  

We will show how to apply Gersten's theorem to give an algorithmic
solution to the problem of determining whether or not a given word is
quasi-positive in $F(X)$.

\section{Quasi-positivity Algorithms for Free Groups}

Let $F(X)$ be the free group with basis $X$.  Let $Y = X \cup X^{-1}$,
and let $Y^*$ be the free monoid on $Y$.  The natural surjection $Y^*
\to F(X)$ is defined by freely reducing, i.e. by removing subwords of
the form $xx^{-1}$ or $x^{-1}x$ for some $x \in X$.  An element of
$F(X)$ is positive if belongs to the positive semigroup $F^+(X)$ consisting
of nonempty reduced words with no occurrence of a letter in $X^{-1}$.
By convention, conjugation is denoted by $a^b =b^{-1} a b$.  As in the
introduction, we define an element $w \in F(X)$ to be quasi-positive if
it admits a factorization of the form $w = w_1^{c_1} \cdots w_k^{c_k}$
such that every $w_i$ is positive.  To simplify the algorithms and the
statements of our results, we also call the identity element a
quasi-positive element.  Since recognizing whether or not a word over $X$
represents the identity element can be solved in linear time, this does
not change the complexity of recognizing quasi-positivity.

The starting point of our investigation is the following:

\begin{lemma} \label{L:abelianize} If $w \in F(X)$ is quasi-positive,
then there exists a factorization of the form $w = x_{i_1}^{c_1} \cdots
x_{i_k}^{c_k}$, where each $x_{i_j} \in X$.  Moreover, if $\Theta: F(X)
\to \mathbb{Z}^X$ is the abelianization map, then the number of
$x_{i_j}$ that is equal to some $x \in X$ is equal to the $x$-coordinate
of $\Theta(w)$.  \end{lemma}

\begin{proof} If $w$ is quasi-positive, then $w$ is equal to a (possibly
empty) word of the form $w_1^{c_1} \cdots w_m^{c_m} \in F(X)$ such that
every $w_i \in F^+(X)$. Using the property that $(ab)^c = a^cb^c$, we
can expand the expression until all of the bases of the exponents are
positive letters.  To see that the number of factors is computed by
$\Theta$, observe that $\Theta(x_{i_j}^{c_j}) = \Theta(x_{i_j})$ adds
one to the $x$-coordinate of $\Theta(w)$ if and only if $x_{i_j}= x$.
\end{proof}

We relate this to Gersten's theorem as follows.  Suppose that $w$ is
quasi-positive.  By Lemma~\ref{L:abelianize}, there exist $x_{i_1},
\dots, x_{i_k} \in X$ and $c_1, \dots c_k \in F(X)$ such that \[w
\cdot (x_{i_k}^{-1})^{c_1} \cdots (x_{i_1}^{-1})^{c_k} = 1.\] 
By Gersten's theorem, there is a spherical cancellation diagram with
$k+1$ boundary circles: one circle is subdivided into $|w|$ arcs and
vertices and these vertices are labeled by the letters in $w$, and the
other $k$ circles each have a single vertex labeled by their respective
negative letter $x_{i_j}^{-1}$.  We refer to the letters $x_{i_1} \dots,
x_{i_k} \in X$ in such a factorization as the base letters.  We will see
in subsection~\ref{S:not_unique} that a quasi-positive element may
admit more than one factorization as a product of conjugates of positive
letters; however, the number and type of the base letters recorded by
$\Theta(w)$ must be the same.

Conversely, we deduce from Gersten's theorem and
Lemma~\ref{L:abelianize} that the question of whether or not a given $w
\in F(X)$ is quasi-positive is equivalent to whether or not there is a
spherical cancellation diagram for at least one sequence of the form
$w,x_{i_1}, \dots, x_{i_k}$, where $k = |\Theta(w)|$, the sum of the
coordinates of $\Theta(w)$, and the number
of $x_{i_j}$ equal to a given $x \in  X$ is the $x$-coordinate of
$\Theta(w)$.  Note that there are only finitely many such sequences. 

It is not immediately clear how to decide for a given $w \in F(X)$
whether or not such a cancellation diagram exists.  The content of
Algorithm~\ref{TestQP} ({\scshape TestQP}) is an effective and algebraic
procedure for deciding whether or not such a diagram exists.  The
content of Algorithm~\ref{FactorQP} ({\scshape FactorQP}) is an
effective and algebraic procedure for writing a word $w$ as a product of
conjugates of positive letters if {\scshape TestQP} has determined that
$w$ is quasi-positive.   The validity of both algorithms is established
in Theorem~\ref{T:algorithms}.

Before describing these algorithms (and later investigating their
computational complexity), we give a few examples to illustrate some
of the difficulties that need to be addressed.

\subsection{First example: a quasi-positive element might be disguised.}

Let $X = \{a,b\}$ and consider the following word over $X$: 

\[b \cdot aba^{-1} \cdot a^2 b^{-1} a^{-1} b a b a^{-2}.\]

We have chosen a word which represents a quasi-positive element of
$F(X)$.  But, if we freely reduce this word and call the corresponding
element $w$, then the fact that $w \in F(X)$ is quasi-positive may be
less apparent: \[w = babab^{-1}a^{-1}baba^{-2}.\]

Since $\Theta(w) = (0,3)$, any factorization of $w$ as a product of
conjugates of letters in $X$ must consist of 3 base letters all of which
are equal to $b$.  Topologically, there must exist a cancellation
diagram for $(w, b, b, b)$.  We invite the reader to find such a diagram
on sphere with 4 boundary circles, one labeled by $w$ and the others
each labeled by the letter $b^{-1}$.

\subsection{Second example: quasi-positive factorizations are not
unique.}  \label{S:not_unique} Let $X = \{a,b\}$ and consider the following word over $X$:
\[a \cdot (bab^{-1}) \cdot b \cdot (ba^{-1} bab^{-1}).\] 

If we freely reduce this word and call the corresponding element $w$,
then the fact that $w \in F(X)$ is quasi-positive is more apparent than
in the previous example: \[w = ababa^{-1}bab^{-1}.\]

We find that $\Theta(w) = (2,2)$.  So, any factorization of $w$ as a
product of conjugates of letters in $X$ must consist of 4 base letters:
two of which are $a$ and two of which are $b$.  Topologically, there
must exist a cancellation diagram for $(w,a, a, b, b)$.  We again invite
the reader to find such a diagram on a sphere with 5 boundary circles,
one labeled by $w$ and the others each labeled by the letters $a^{-1}$,
$a^{-1}$, $b^{-1}$, and $b^{-1}$, respectively.  Note that a solution is
not unique.  Indeed, the following are different factorizations of $w$
as a product of conjugates of positive letters: \[w = a \cdot b \cdot
aba^{-1} \cdot bab^{-1} = a \cdot b \cdot a \cdot ba^{-1}bab^{-1}.\]

\subsection{Third example: the abelianization is not a complete
invariant.}  \label{third_abelianization_not_complete}  If $w \in F(X)$
is quasi-positive, then $\Theta(w)$ must have non-negative coordinates.
But the converse is not true, as the following example shows: let $w =
[a,b]b = a^{-1}b^{-1}abb$.  We have that $\Theta(w) = (0,1)$, but there
is no spherical cancellation diagram for $w$ with two boundary circles
such that one is labeled by $w$ and the other by $b^{-1}$.  

\subsection{The Test for Quasi-Positivity algorithm.}

The algorithm {\scshape TestQP}$(w)$ (described below in
Algorithm~\ref{TestQP}) accepts a word $w$ as input and returns true if
$w$ is quasi-positive and false if it is not.  The reader is advised to
read this algorithm in tandem with the example in
subsection~\ref{fourth_example}.  The algorithm can be modified so that
if $w$ is quasi-positive, then it also returns a labelled binary tree
$T$ that encodes how it was determined that $w$ is quasi-positive.  This
tree can then be used as input for the algorithm {\scshape
FactorQP}$(T)$, and a factorization of $w$ as a product of conjugates of
positive letters will be returned.  A detailed example of this
procedure appears in subsection~\ref{factor_example}.

\begin{algorithm} \caption{Test for Quasi-Positivity} \label{TestQP}

\begin{algorithmic}[1] 

\Procedure{TestQP}{$w$} 
\Comment{The input word $w = w_1 \cdots w_n$ is stored in a list so that
$w[i] = w_i$ and each $w_i \in X \cup X^{-1}$.  We use the notation
$w[i:j] = [w_i, \dots, w_{j-1}]$ for $1 \leq i < j \leq n$.}

\For {$i \gets 1$ to $n$} \Comment{Find the initial negative letter.}

	\If {$(w[i] \in X^{-1})$} \State ${\rm double} \gets ww$
\Comment{Double $w$ to simulate a cyclic word.  Here $ww$ denotes the
concatenation of the list $w$ with itself.}

		\For {$j \gets 1$ to $n-1$}
\Comment{Search for a matching positive letter $x \in X$ that is the
inverse of the negative letter $w[i] = x^{-1}$.} 

			\If {$({\rm double}[i+j] = w[i]^{-1})$}
\Comment{The pair of matching inverse letters $x^{-1}, x$ subdivides 
the cyclic word into two subwords: the word $w_L$ to the Left of $x^{-1}$
and to the right of $x$ and the word $w_R$ to the Right of $x^{-1}$ and
to the left of $x$.  In a cyclic word, these will wrap around.}

					\State $w_L \gets {\rm
double}[i+1:i+j]$ 

					\State $w_R \gets {\rm
double}[i+j+1: n+i]$ 

					\If {({\scshape TestQP}($w_L$)
\, {\bf and} \, {\scshape TestQP} ($w_R$)) } 
\Comment{Recursively test $w_L$ and $w_R$.  If both are quasi-positive,
then the matching inverse letter is ``good''.  If either $w_L$
or $w_R$ is not quasi-positive, then lines 11 and 12 will be called and 
the FOR loop on line 5 will increment $j$ and continue to search for a 
good matching inverse letter.}
						\State {\bf return true}
					\EndIf 
			\EndIf 
		\EndFor \Comment{If line 13 is called, then there is no 
good matching inverse letter for the negative letter $x^{-1}$.  By
Lemma~\ref{L:one_negative_letter_suffices}, we deduce that $w$ is not
quasi-positive.} 
		\State {\bf return false} 
	\EndIf 
\EndFor \Comment{If line 16 is called, then $w$ is positive or empty.} 

\State {\bf return true}

\EndProcedure 

\end{algorithmic} 

\end{algorithm}

\begin{algorithm} \caption{Factor Quasi-Positive element}
\label{FactorQP} \begin{algorithmic}[1] \Procedure{FactorQP}{$Tree$}
\Comment{The input $Tree$ is a rooted binary tree.  Every degree one
vertex contains a positive word.  Every vertex of degree greater than
one contains a quasi-positive word $w$ and a positive letter $x$ coming
from good pair $x^{-1},x$ contained in $w$.  The two children of this
vertex contain the words $w_L$ and $w_R$, respectively.}

\State Number the vertices of $Tree$ by $1, 2, \dots$ using a Breadth
First Search.  \State {\bf repeat} 

\State Find vertices (necessarily of degree one) with the two largest
numbers.  Let $w_L$ and $w_R$ be the words they contain.  

\State Let $w$ and $x$ denote the word and positive letter,
respectively, contained in the parent vertex of $w_L$ and $w_R$.

\State Since the tree was constructed from a call to {\scshape
TestQP}$(w)$, we know that $w$ and $x^{-1} w_L x w_R$ are equal as
cyclic words.  Cyclically permute the letters of $x^{-1} w_L x w_R$
until it equals $w$ letter for letter.

\State Using one of the cases of Lemma~\ref{L:cyclingPCP}, write $w$ as
a product of conjugates of positive letters.  Overwrite the contents of
the parent vertex with this factorization.

\State Delete both children of the parent node.

\State Repeat using the next pair of vertices with the largest two
numbers.

\State {\bf until} all nodes visited

\State {\bf return} the factorization of $w$ contained in the root
vertex

\EndProcedure

\end{algorithmic} \end{algorithm}

\subsection{Fourth example: the {\scshape TestQP} algorithm in
practice.} \label{fourth_example} Consider the following element of
$F(X)$: \[w = babab^{-1}a^{-1}baba^{-1}a^{-1} = w_1 w_2 \cdots w_{10}
w_{11},\] where $X = \{a,b\}$ and each $w_i$ is a letter, i.e. an
element of $Y = X \cup X^{-1}$.  We describe the steps that result from
calling {\scshape TestQP}$(w)$:


\begin{enumerate}

\item The procedure (line 1) is called with input $[w_1, \dots,
w_{11}]$.  The initial negative letter $w_5 = b^{-1}$ is found (line 2).
The list \[{\rm double} = [w_1, \dots, w_{11}, w_1, \dots, w_{11}] =
[d_1, \dots, d_{22}]\] is constructed to facilitate our need to regard 
$w$ as a cyclic word (without needing to speak of indices modulo $|w|$).

\item The next (starting with $w_6$) occurrence of the matching 
inverse letter is found: $w_7 = b$ is the inverse of $w_5$.  The
cyclic word $w$ is then split into two subwords by deleting the
pair $w_5, w_7$ of inverse letters: let $w_L = a^{-1}$ (the left
subword, meaning lying to the left of the initial negative letter)
and $w_R = aba^{-1}a^{-1}baba$ (the right subword, which wraps around
since $w$ is a cyclic word).  

\item Recursion is then used: the algorithm is called with input $w_L$
(line 9).  This returns false since, although a negative letter, namely
$a^{-1}$ is found in $w_L$, there is no matching inverse letter.  The
algorithm is not called on $w_R$ since the compiler will recognize that
the condition inside the IF statement in line 9 is false.

\item The FOR loop (line 5) finds the next matching inverse letter
(starting the search at $w_7$): $w_9 = b$.  As above, the cyclic word
$w$ is split into two subwords: let $w_L = a^{-1} b a$ and let 
$w_R = a^{-1} a^{-1} baba$.  The algorithm is called on $w_L$.  This
time, the letter $w_6 = a^{-1}$ is found, a matching inverse letter,
$w_8 = a$ is found and $w_L$ is split into two subwords: $w_{LL} = b$
and $w_{RL} = 1$ (the empty word).  Recursively, the algorithm is
called on $w_{LL}$ and true (line 16) since $w_{LL}$ is positive.
The same happens for $w_{RL}$ since it is empty. 

\item At this point, looking forward to constructing a factorization,
the following information will be recorded in a rooted binary tree (see
Figure~\ref{F:tree}).  The root will contain the word $w$ and the letter
$b$ (since $b^{-1}$ was found in $w$); the left child of the root will
contain the word $w_L$ and $a$ (since $a^{-1}$ was found in $w_L$).  The
two children of $w_L$ are vertices continuing $w_{LL} = b$ and $w_{RL} =
1$, respectively.  These steps are not included in the algorithm so that
the exposition is more clear.

\item Next, the algorithm is called on $w_R = a^{-1}a^{-1}baba$.  The
first negative letter, $w_{10} = a^{-1}$ is found.  The next matching
positive letter $w_2 = a$ is found; this results in the subwords $w_{LR} =
a^{-1}b$ (in line 7) and $w_{RR} = ba$.  When the algorithm is called on
$w_{LR}$ it will find no positive letter to match $a^{-1}$.  Therefore,
the algorithm will continue (line 5) to search $w_R$ for a good match
for the negative letter $a^{-1}$.  Indeed, $w_4 = a$ is a good match
since $w_{LR} = a^{-1}bab$ and $w_{RR} = 1$ are quasi-positive.  The 
negative letter $w_{11} = a^{-1}$ is found in $w_{LR}$ and matched
with $w_2 = a$ and the resulting subwords are $w_{LLR} = b$ and 
$w_{RLR} = b$.  These are both positive.  These two subwords and the 
letter $a$ are recorded the binary tree as shown in
Figure~\ref{F:tree}.

\item Finally, $w_{RR} = 1$ is recognized as empty.  Tracing
back through the recursive calls, we find that both conditions in the
test on line 9 are satisfied at each step.  The result is that the
algorithm will return true (line 10) for the input word $w$.

\end{enumerate}

{\bf Remark:}  The {\scshape TestQP} algorithm tests whether
or not the input word represents a quasi-positive element of a free
group.  We have chosen to regard a word that represents the identity as
quasi-positive since it simplifies the algorithm.  However, it is
straightforward to recognize whether or not a word represents the
identity and doing so will not change the worst-case runtime of the
algorithm.

\begin{figure} \begin{tikzpicture} \Tree
[.$w=babab^{-1}a^{-1}baba^{-1}a^{-1},b$ [.$w_L=a^{-1}ba,a$ [.$w_{LL}=b$
] [.$w_{RL}=1$ ] ] [.$w_R=a^{-1}a^{-1}baba,a$ [.$w_{LR}=a^{-1}bab,a$
[.$w_{LLR}=b$ ] [.$w_{RLR}=b$ ] ] [.$w_{RR}=1$ ] ] ] \end{tikzpicture}
\caption{The tree $T$ constructed in the process of verifying that $w =
babab^{-1}a^{-1}baba^{-1}a^{-1}$ is quasi-positive.  Each leaf of $T$
contains a positive word.  Other nodes contain a word and a positive
letter $x \in X$ corresponding to a good pair $x^{-1},x$ in $w$.}
\label{F:tree} \end{figure}
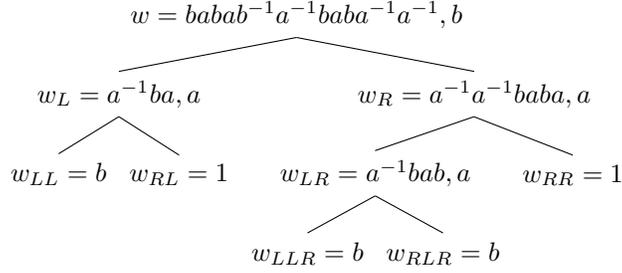

\subsection{Fifth example: the {\scshape FactorQP} algorithm in
practice.}  \label{factor_example} We step through a call of {\scshape
FactorQP}$(T)$, where $T$ is the binary tree in Figure~\ref{F:tree}
corresponding to the example in subsection~\ref{fourth_example}

\begin{enumerate}

\item The vertices of the tree $T$ are numbered from 1 (the root) to 9
(the vertex containing $w_{RLR} = b$ using a breadth first search (line
2).

\item The vertices with the two largest numbers, 8 and 9, contain
$w_{LLR} = b$ and $w_{RLR} = b$, respectively (line 4).

\item The parent vertex (with number 6) contains the word $w_{LR} =
a^{-1}bab$ and the letter $a$ (line 5).

\item We construct $a^{-1}w_{LLR}aw_{RLR} = a^{-1} bab$.  It so happens
that this is equal to $w_{LR}$ and so no cycling is required (line 6).

\item We write the factorization $a^{-1}ba \cdot b$ of $w_{LR}$ to the
vertex with the number 6 and delete the vertices with numbers 8 and 9
(lines 7 and 8).  By abuse of notation, we denote this factorization by
$w_{LR}$.

\item We repeat.  The next two largest numbered vertices are vertex 6
and vertex 7 which contain $w_{LR} = a^{-1}ba \cdot b$ and $w_{RR} = 1$,
respectively (line 5); their parent vertex contains $w_R$ and the letter
$a$.

\item As we step through lines 6--9, we find that \[a^{-1}w_{LR}aw_{RR}
=  a^{-1}(a^{-1}ba \cdot b) a \cdot 1\] is equal to $w_R$ letter for
letter (no cycling required); however expressing this as a product of
conjugates of positive letters requires some care, as detailed in
Lemma~\ref{L:cyclingPCP}.  The correct factorization is the following:
\[b^{aa} \cdot b^{a} = a^{-1}a^{-1}baa \cdot a^{-1}ba.\] This
factorization is written to vertex 3 and denoted by $w_R$.  

\item The next two highest numbered vertices contain $w_{LL} = b$ and
$w_{RL} = 1$, respectively.  Using the letter $a$ of the parent vertex,
we obtain the factorization $a^{-1}ba$ of $w_L$; no cycling is required.

\item Finally, we consider vertex 2 (containing the factorization
$a^{-1}ba = w_L$) and vertex 3 (containing the factorization
$a^{-1}a^{-1}baa \cdot a^{-1} b a = w_R$).  The parent vertex (the root)
contains the letter $b$ and so we try to match $w$ letter for letter
with the word $b^{-1} w_L b w_R$.  Cycling is required in this case.
Moreover, expressing $w$ as a product of conjugates of positive letters
requires some care.  We have that  \[b^{-1} w_L b w_R = b^{-1} (a^{-1}
ba) b \cdot (a^{-1}a^{-1}baa \cdot a^{-1}ba).\] After freely reducing
the above word, it is equal as a cyclic word to $w$.  We cyclically
permute the letters and obtain the following factorization as detailed
in Lemma~\ref{L:cyclingPCP}: \[b \cdot b^{a^{-1}} \cdot
b^{aba^{-1}a^{-1}}.\] The above factorization is the factorization of
$w$ as a product of conjugates of positive letters that is returned by
{\scshape}{FactorQP}$(w)$.

\end{enumerate}

\section{Validity of the algorithms}

\begin{lemma} \label{L:one_negative_letter_suffices} Suppose $w = y_1
\cdots y_{n}$, where each $y_i \in X \cup X^{-1}$.  If $w$ is
quasi-positive (or $w$ represents the identity) and some $y_i \in
X^{-1}$, then there exists a $y_j$ such that $y_j = y_i^{-1}$ and both
$w_L$ and $w_R$ are quasi-positive, where \[y_{i} \cdots y_n y_1 \cdots
y_{i-1} = y_iw_Ly_jw_R.\] \end{lemma}

\begin{proof} If $w$ is quasi-positive and $w_i \in X^{-1}$, then by
Gersten's theorem there is a spherical cancellation diagram with an arc
joining the vertex with label $y_i$ on the $w$-curve to a vertex with
label $y_i^{-1}$ on the $w$-curve.  If we cut this diagram along this
arc, we obtain two spherical cancellation diagrams: one for $w_L$ and
one for $w_R$.  Therefore, both $w_L$ and $w_R$ are quasi-positive.
\end{proof}

The significance of the above is that in Algorithm~\ref{TestQP}
({\scshape TestQP}), if a letter in $X^{-1}$ is read from the input word
$w$, then we can search for a matching inverse letter and run the
algorithm again on the two subwords $w_L$ and $w_R$.  If neither is
found to be quasi-positive, then we proceed to the next
inverse letter and repeat.  If all matching inverse letters are exhausted
and no pair of subword is quasi-positive, then we can conclude that $w$
is not quasi-positive, whereas a priori we might have to check the
next letter of $w$ in $X^{-1}$ to find a good pair of inverse letters.

\begin{theorem} \label{T:algorithms} 
\leavevmode

\begin{enumerate}

\item The algorithm {\scshape TestQP} correctly determines whether or
not an input word $w$ is quasi-positive.  

\item Suppose that $w$ is quasi-positive.  Let $T$ be the rooted binary
tree encoding the good pairs of inverse letters and corresponding
subwords returned by a modification of {\scshape TestQP} as described in
subsections~\ref{fourth_example} and \ref{factor_example}.  Then
{\scshape FactorQP}$(T)$ will return a factorization of $w$ as a product
of conjugates of positive letters.

\end{enumerate} \end{theorem}

\begin{proof} If $w$ is quasi-positive, then there exists a cancellation
diagram with one boundary cycle labeled by $w$ and the remaining cycles
each labeled by one letter in $X^{-1}$.  As described in
Lemma~\ref{L:abelianize}, the number of these cycles and the number
having a given negative letter is determined by the abelianization
$\Theta(w)$.  We argue that the algorithm {\scshape TestQP} recognizes
that $w$ is quasi-positive and that the algorithm {\scshape FactorQP}
returns a factorization as a product of conjugates of positive letters.

The argument is by induction on the length of $w$.  The statement is
vacuously true for the trivial word.  Suppose $w$ has length $L > 0$.
If the cancellation diagram has no arcs joining vertices of the boundary
cycle labeled by $w$ to itself, then $w$ is positive and this is
detected by line 16 of Algorithm~\ref{TestQP}.  Otherwise there are
vertices labeled by $x$ and $x^{-1}$ for some $x \in X$, and there is an
arc in the cancellation diagram joining $x$ to $x^{-1}$.  The word $w$
admits a factorization as $w = t x^{-1} u x v$.  By cutting along this
arc we obtain two spherical cancellation diagrams, one for the word $u$ and
one for the word $vt$.  Since $u$ and $vt$ are shorter quasi-positive
words (as in the proof of Lemma~\ref{L:one_negative_letter_suffices}),
by induction, the algorithms correctly identify them as quasi-positive
and return factorizations of the desired type.  Lemma~\ref{L:cyclingPCP}
contains further explanation of how the factorization is obtained from
the recursion.

Conversely, we argue that the algorithm {\scshape TestQP} can be used to
construct a spherical cancellation diagram.  Suppose that the algorithm
returns true for an input word $w$.  If $w$ is positive, it is clear
that a spherical cancellation diagram exists.  Otherwise, {\scshape
TestQP} find a good pair of inverse letters $x^{-1}, x$.  Join the
corresponding vertices by a proper tame arc on a sphere with
$|\Theta(w)| + 1$  boundary circles of the correct type as discussed
previously.  By induction on the length of the input, spherical
cancellation diagrams exist for the subwords $w_L$ and $w_R$
corresponding to this pair.  These spherical diagrams can be glued to
form a spherical cancellation diagram for $w$.  By Gersten's theorem,
$w$ is quasi-positive.   This proves that {\scshape TestQP} correctly
identifies whether or not an input word is quasi-positive.  \end{proof}

We saw in subsection~\ref{fourth_example} that finding a factorization
of the input word as a product of conjugates of positive words, is
delicate.  The following lemma explains how Algorithm~\ref{FactorQP}
({\scshape FactorQP}) finds such a factorization of $w$ given that $w$
contains a good pair $a, a^{-1}$ of inverse letters and quasi-positive
subwords $w_L$ and $w_R$ with given factorizations as a product of
conjugates of positive elements.

\begin{lemma} \label{L:cyclingPCP} Let $F(X)$ be free with basis $X$.
Let $w \in F(X)$.  Let $\bar{w}$ be the cyclic reduction of $w$ so that
$w = v^{-1} \bar{w} v$ and $2|v| + |\bar{w}| = |w|$ and $v \in F(X)$.
Suppose that $w_L$ and $w_R$ are words over $X$, each expressed as a
product of conjugates of positive letters.  Let $a \in X$.  Suppose that
$\bar{w}$ and $w_L^a w_R$ are conjugate in $F(X)$ via some $u \in F(X)$:
\[\bar{w} = (w_L^a w_R)^u.\] Then \[w = \bar{w}^v = w_L^{auv} \cdot
w_R^{uv}\] is a factorization of $w$ as a product of conjugates of
positive letters.  Moreover, elements $u$ and $v$ can be effectively
determined from the words $w$ and $a^{-1}w_La w_R$ and the assumption
that these two words define conjugate elements in $F(X)$.  \end{lemma}

\begin{proof} The word $v$ can be found in linear time by successively
checking whether or not the initial and terminal letters of $w$ are equal
and $|w| > 1$.  The element $\bar{w}$ is defined as the freely reduced
word obtained from $vwv^{-1}$.  Suppose that $w$ and $w_L^aw_R$ are
conjugate.  Freely and cyclically reduce $w_L^aw_R$ to obtain a word
$z$.  It follows that $\bar{w}$ and $z$ are related by a cyclic
permutation.  By the Knuth-Pratt-Morris string matching algorithm, we
can find a subword $u$ of $\bar{w}$ such that $\bar{w} = (w_L^aw_R)^{u}$
in $F(X)$.  It then follows that \[w = \bar{w}^v = (w_L^aw_R)^{uv} =
w_L^{auv} w_R^{uv}.\] By distributing the exponents across the given
factorizations of $w_L$ and $w_R$, a factorization of $w$ as a product
of conjugates of positive letters is obtained.  \end{proof}

\section{Worst-case runtime analysis of the {\scshape TestQP} algorithm}

The {\scshape TestQP} algorithm takes a word of length $n$ as input and
returns yes if the word represents a quasi-positive element and no if it
does not.  The algorithm works as follows: \bigskip

\begin{tabular}{ll} 1. input word of length $n$&		(let
$f(n)$ be worst-case runtime)\\ 2. find the first negative letter&
(runtime $Cn$)\\ 3. find a matching positive letter&	(runtime
$C(n-1)$)\\ 4. check whether subwords are quasi-positive& (runtime $f(k)
+ f(n-k-2)$)\\ 5. repeat 3 and 4 until no more matches&		(leads
to a sum of $f(k) + f(n-k-2)$) \end{tabular} \bigskip

From the above, we see that \[f(n) \leq \sum_{k=0}^{n-2} (f(k) + f(n-k))
+ C(2n-1).\] We can rewrite this as follows, using the fact that
$f(0)=0$, by replacing the constant $C$ by a larger constant, and by
re-writing the sum: \[f(n) \leq 2\left(\sum_{k=1}^{n-2} f(k)\right) +
Cn\] Assuming equality, we have following recurrence relation: \[f(n+1)
- f(n) - 2f(n-1) = C.\] This is a linear first order non-homogeneous
recurrence relation with characteristic equation $x^2 - x - 2 = 0$.  The
following general solution is obtained by observing that $f(0) = 0$
and $f(1) = C$ and that $f(n) = -C/2$ is a particular solution: \[f(n) =
\frac{2C}{3}2^n - \frac{C}{6}(-1)^n - \frac{C}{2}.\] Therefore, the
{\scshape TestQP} algorithm has exponential worst-case runtime.

\section{A quasi-positive word with long {\scshape TestQP} runtime}
\label{MainExample} Consider the following word over $X = \{a,b\}$:


\[u_k = ab[a,b]^kb^{k-1}.\]

\subsection{The word $u_k$ is quasi-positive}

If $k=1$, then $u_k = u_1 = ab[a,b] = aba^{-1} \cdot b^{-1}ab \cdot b$,
and so is quasi-positive. 

For $k > 1$, we have the following: \begin{equation} \begin{aligned} u_k
&= ab[a,b]^k b^{k-1}\\ &= aba^{-1} \cdot b^{-1} (ab [a,b]^{k-1} b^{k-2}
)b\\ &= b^{a^{-1}} \cdot u_{k-1}^{b}, \end{aligned} \end{equation} and
so $u_k$ is quasi-positive by induction on $k$ since $QP$ is closed
under products and under conjugation.

\subsection{Subwords of $u_k$}

\begin{remark} The initial $a^{-1}$ in $u_k$ cannot pair with any $a$
except the first $a$, because if it paired with some $a$ belonging to 
the product of commutator factors, then the left subword $w_L$ would have 
negative $b$-coordinate in $\Theta(w_L)$.  \end{remark}

\begin{remark} An element of the form $[a,b]^k b^j$, where $k \geq 1$
and $j \geq 0$ is never quasi-positive.  For if it were quasi-positive,
then in a spherical cancellation an arc pairing some $a^{-1}$ with some
$a$ would leave a subword of the form $b^{-1}ab[a,b]^na^{-1}b^{-1}$,
which abelianizes to $(0,-1)$ and so is not quasi-positive.
\end{remark}

\begin{lemma} The element $ab[a,b]^kb^{k-2}$ is not quasi-positive.
\end{lemma}

\begin{proof} The factor $[a,b]^k$ is not quasi-positive.  Therefore,
any spherical cancellation diagram for the word $ab[a,b]^kb^{k-2}$ must
have an arc joining the initial $a$ to some $a^{-1}$ in a commutator.
(Otherwise, deleting the initial $a$ would produce a
quasi-positive word.)  The following left subword results: $w_L =
b[a,b]^i$, for some $i \geq 0$.  But this word is not quasi-positive by
Remark 2.  \end{proof}

\begin{lemma} The only way to realize $b^{-1}u_kb$ as a quasi-positive
element is to pair the outer pair $b^{-1}$, $b$ in a spherical
cancellation diagram.  \end{lemma}

\begin{proof} We argue by
contradiction.  Assume that the initial $b^{-1}$ 
can be joined by an arc to some $b$ in $u_k$.  There are three cases to
consider.  

Case 1:  If the initial $b^{-1}$ is joined to the first $b$ in $u_k$,
then the right subword is $[a,b]^kb^{k-1}$, which is not quasi-positive
by Remark 2 above. 

Case 2:  If the initial $b^{-1}$ is joined to some $b$ in a commutator
factor of $u_k$, then  the right subword has the form $[a,b]^jb^{k}$, 
which is not quasi-positive by Remark 2 above. 

Case 3:  Assume the initial $b^{-1}$ is joined to some $b$ in the last
factor of the form $b^{k-1}$ of $u_k$, but not to the last such b. Then,
the left subword has the form $ab[a,b]^{k-1}b^j$ for some $j < k-2$, and
so is not quasi-positive by Lemma 1 above.  \end{proof}

\subsection{Recursive structure of {\scshape TestQP}$(u_k)$}

When {\scshape TestQP}$(u_k)$ is called, the algorithm attempts to pair 
the initial $a^{-1}$ with each occurrence of $a$, of which there are
$k+1$.  By Remark 1 above, only
the first $a$ will be a good match.  Thus, the algorithm has called
itself $k+1$ times.  The resulting left subword is $b^{-1}u_{k-1}b$.  By
Lemma 2 above, only the final $b$ will pair with the $b^{-1}$.  When
this good match is made, the resulting left subword is $u_{k-1}$,  
revealing the recursive structure of the work $u_k$. 

The above analysis suggests that the runtime of {\scshape
TestQP}$(u_k)$ is exponential in $k$.  On the other hand, the above
analysis shows that if the algorithm is modified so that the input is
first cyclically reduced and the abelianization is computed and checked,
then the runtime changes to quadratic.  In more
detail, suppose that the {\scshape TestQP} algorithm is modified so that
in between lines 1 and 2, it first cyclically reduces the input and
then computes $\Theta(w)$ (both of which can be done in linear time).
A line is added so that if $\Theta(w)$ has a negative coordinate, then 
the algorithm immediately returns false.  After these additions, 
{\scshape TestQP}$(u_k)$ will run in quadratic time.  This is because
when we check to see if a positive letter $a$ is a good match for the
initial occurrence of $a^{-1}$, we can reject the bad matches using the
abelianization map.  So, there will be $k$ checks
(with linear run time) and one final check which finds a good match.  The
resulting subwords are $w_L = b^{-1}u_{k-1}b$ and $w_R = b$.  The runtime
(after cyclically reducing $w_L$) will be linear in $k-1$ by induction.
The total runtime will be on the order of $k + (k-1) +  \cdots + 1$, which
is quadratic in $k$.

\begin{center} \begin{figure} \begin{tikzpicture}[scale=.4] \path (-6,0)
node[anchor = east] {$b$} (-5,2) node[anchor = south east] {$b$} (-2,5)
node[anchor = south east] {$a$} (0,6) node[anchor = south] {$b$} (2,5)
node[anchor = south west] {$a^{-1}$} (5,2) node[anchor = south west]
{$b^{-1}$} (6,0) node[anchor = west] {$a$} (5,-2) node[anchor = north
west] {$b$} (2,-5) node[anchor = north west] {$a^{-1}$} (0,-6)
node[anchor = north] {$b^{-1}$} (-2,-5) node[anchor = north east] {$a$};
\draw[thick, rounded corners = 8pt] (-6,0) -- (-5,2) -- (-2,5) -- (0,6)
-- (2,5) -- (5,2) -- (6,0) -- (5,-2) -- (2,-5) -- (0,-6) -- (-2,-5) --
(-5,-2) -- cycle; \node at (-6,0) [circle, inner sep = 2pt, draw =
black, fill = black] {}; \node at (-5,2) [circle, inner sep = 2pt, draw
= black, fill = black] {}; \node at (-2,5) [circle, inner sep = 2pt,
draw = black, fill = black] {}; \node at (0,6) [circle, inner sep = 2pt,
draw = black, fill = black] {}; \node at (2,5) [circle, inner sep = 2pt,
draw = black, fill = black] {}; \node at (5,2) [circle, inner sep = 2pt,
draw = black, fill = black] {}; \node at (6,0) [circle, inner sep = 2pt,
draw = black, fill = black] {}; \node at (5,-2) [circle, inner sep =
2pt, draw = black, fill = black] {}; \node at (2,-5) [circle, inner sep
= 2pt, draw = black, fill = black] {}; \node at (0,-6) [circle, inner
sep = 2pt, draw = black, fill = black] {}; \node at (-2,-5) [circle,
inner sep = 2pt, draw = black, fill = black] {};
     
     \draw[thick, red] (-6,0) .. controls (-1,-1) and (-1,-1) .. (0,-6);
\draw[thick, red] (-5,2) -- (5,2); \draw[thick, red] (-2,5) .. controls
(-1,2) and (1,2) .. (2,5); \draw[thick, red] (6,0) .. controls (0,2) and
(-2,0) .. (2,-5);
     
     \draw[thick, blue] (0,4) circle [radius = .9]; \draw[thick, blue]
(3,-1) circle [radius = .9];  \draw[thick, blue] (-2.5,-2.5) circle
[radius = .9];
     
     \node at (0,4) {$b^{-1}$}; \node at (-2.5,-2.5) {$a^{-1}$}; \node
at (3,-1) {$b^{-1}$};

     \node at (0,4.9) [circle, inner sep = 2pt, draw = black, fill =
black] {}; \node at (-2.5,-3.4) [circle, inner sep = 2pt, draw = black,
fill = black] {}; \node at (3.9,-1) [circle, inner sep = 2pt, draw =
black, fill = black] {};
     
\end{tikzpicture} \caption{A spherical cancellation diagram for $u_2 =
ab[a,b]^2b$.  We can recover a factorization as a product of conjugates
of positive letters as follows: choose, as a base point, the midpoint of
the edge joining the first and last letters of $u_2$; then join each
vertex on an inner circle by a transverse arc to the base point the
letters joined by the arcs record the conjugating elements and the inner
circles, (the inverse of) the base letters.  In this case: $u_2 =
b^{a^{-1}}b^{ab^{-1}} (a^{-1})^{b^{-2}}$.} \end{figure}
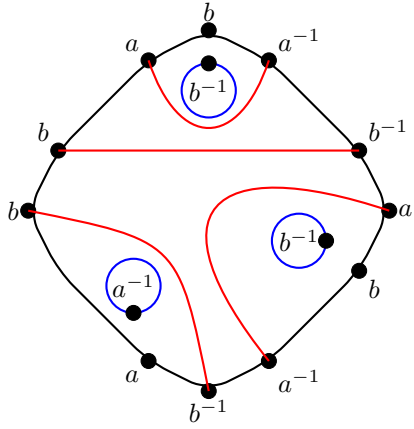 \end{center}

\section{Brute force algorithm for checking quasi-positivity in a free
group}

Given $w \in F(X)$, we can search for a
factorization as a product of conjugates of positive basis elements by
the following
brute force approach.  From $\Theta(w)$, we know exactly the number of
each basis element that will appear in a factorization as a base $w_i$
as below: \[w = w_1^{c_1} \cdots w_k^{c_k},\] where $w_1, \dots, w_k \in
X$ and $k$ is the sum of the coordinates of $\Theta(w)$.  There are only
finitely many such $c_i$.  To see this, observe that if such a
factorization exists, then there is a spherical cancellation diagram
with outer circle labelled by $w$ and inner circles labelled by
$w_1^{-1}, \dots, w_k^{-1}$.  Select the midpoint of the first edge (the
one incident to the vertex corresponding to the initial letter of $w$)
of the boundary circle labeled by $w$ as a base point.  If we join the
base point to the vertex on $w_i$ by an arc transverse to the arcs in the
diagram, then by reading the labels of the arcs that are crossed, we can
write down the letters of $c_i$.  Because the arcs of the diagram join
pairs of vertices of $w$ with inverse labels or join vertices of
positive letters of $w$ to a corresponding inverse letter on an inner
boundary circle, we see that \[|c_i| \leq \frac{|w| - k}{2}.\] Thus,
there are only finitely many such conjugating elements.  And we can
check using the linear time solution to the word problem in $F(X)$
whether or not $w$ is equal to any expression of this form.  This gives
an exponential bound in $|w|$ to the complexity of deciding if $w$
is quasi-positive.

\bigskip

{\bf Remark:} One can ask if there are decision problems about finitely
presented groups which have complexity greater than exponential time.
The following construction shows this is possible for countable groups.
Let $S$ be a recursive subset of $\mathbb{Z}$.  For each $n \in
\mathbb{Z}$, let $G_n = \langle n \mid n \rangle \cong \{1\}$ if $n \in
S$ and $G_n = \langle n \mid \quad \rangle \cong \mathbb{Z}$ if $n
\notin S$.  Let $G$ be the free product of the $\{G_n\}$.  The decision
problem of whether or not a word over $\mathbb{Z}$ represents the
trivial element of $G$ has complexity at least that of the complexity of
deciding whether or not an integer belongs to $S$.  Since $G$ is
recursively presented, it is a subgroup of a finitely presented group.
The membership problem for this subgroup is a decision problem about a
finitely presented group which is at least as complex as the decision
problem for an integer belonging to $S$.

\subsection{Brute force approach applied to $u_k$}

If the brute force approach described above is applied to the word
$u_k$ in section~\ref{MainExample}, where $|u_k| = 5k+1$, we see that 
one needs to determine whether or not $u_k$ is equal (as a cyclic word) to
some word of the form \[a^{c_0}b^{c_1} \cdots b^{c_k} = a^{c_0}b^{c_1
\cdots c_k},\]
where each $|c_i| \leq (|u_k| - (k+1)/2 = 2k$.  This is a very large
collection of words.  The ball of radius $r \geq 1$ in the rank 2 free
group is $1 + 2(3^r - 1)$.

{\bf Remark:} Note that checking
equality in the free group is basically free since the size of the set
$\{c_i\}$ is growing like $2^k$.

\section{Orevkov's method} \label{S:OrevkovMethod}
Stepan Yu.~Orevkov described an algorithm which recognizes quasi-positive
elements of the free group in~\cite{Orevkov_MR2056762}
We describe his methods here and compare them to the topological
methods used in this paper.

\subsection{Quasi-positivity in a free group}
Orevkov recursively defines a regular bracket
structure (RBS) as follows: the empty word and $*$ are RBSs; if $a$
and $b$ are RBSs, then so are $[a]$ and $ab$.  Thus, the set of RBSs,
denoted $\mathcal{R}$, is a subset of the free monoid on the alphabet
$\{[,*,]\}$.  The element $*$ is called star, and the elements $[$ and
$]$ are called brackets.  A pair of brackets match if they arise from
the rule $a \in \mathcal{R} \to  [a] \in \mathcal{R}$.

Suppose that $w = x_1^{\epsilon_1} \cdots x_n^{\epsilon_n}$ is a word 
over $X$, where $\epsilon_i \in \{\pm 1\}$ for $i=1, \dots, n$.  A RBS
$u_1 \cdots u_n \in \mathcal{R}$ is said to agree with $w$ provided

\begin{enumerate}[nosep] 

\item[(1)] if $u_i = *$, then $x_j \in X$, i.e. stars match positive
letters, and 

\item[(2)] if $u_i$ and $u_j$ are matching brackets, then $x_i$ and
$x_j$ are inverse letters.  

\end{enumerate}

\smallskip

We can now state Orevkov's theorem:

\begin{theorem*}[Orevkov, Proposition 1.1 in~\cite{Orevkov_MR2056762}]
\label{T:OrevkovQPFree} 

A word $w$ over $X$ represents a quasi-positive element of the free group
$F(X)$ if and only if there is a regular bracket structure which agrees
with $w$.

\end{theorem*}

Here is a proof of Orevkov's theorem using the methods of this paper.
Suppose that $w = x_1^{\epsilon_1} \cdots x_n^{\epsilon_n}$ represents a
quasi-positive element of $F(X)$.  Then there is a spherical
cancellation diagram for $w$, $\Theta(w)$, where $\Theta:F(X) \to
\mathbb{Z}^{X}$ is the abelianization homomorphism.  If $w$ is positive,
then $u = * \cdots *$ agrees with $w$.   Otherwise, we have that $w =
ax^{\epsilon}bx^{-\epsilon}c$ such that the letters $x^{\pm \epsilon}$
are joined by an arc in the diagram.  Cutting along this arc, we have that
both $b$ and $ca$ are quasi-positive.  We need the following:

\begin{lemma} \label{L:cycleRBS} 

If a word over $X$ agrees with some RBS, then so does any cyclic
conjugate of this word.

\end{lemma}

\begin{proof} 

Suppose $u = u_1 \cdots u_n \in \mathcal{R}$ agrees with the
word $w = w'x^\epsilon$.  If $u_n = *$, then $u_n u_1 \cdots u_{n-1}$ 
agrees with $x^\epsilon w$.  If $u_n = ]$ and $u_i = [$ is the matching
bracket, then $[u_1 \cdots u_{i-1}]u_{i+1} \cdots u_{n-1}$ is an RBS
that agrees with $x^\epsilon w$.

\end{proof}

Continuing with the proof of Theorem~\ref{T:OrevkovQPFree}, we use
induction on the number of arcs, to conclude that there is an RBS $p$ 
which agrees with $b$ and an RBS $q$ which agrees with $ca$.  The
RBS $[b]q$ agrees with $x^\epsilon b x^{-\epsilon}ca$.  By
Lemma~\ref{L:cycleRBS}, there is an RBS which agrees with $w$.

Conversely, suppose that $w = x_1^{\epsilon_1} \cdots w_n^{\epsilon_n}$
is a word over $X$ and $u = u_1 \cdots u_n$ is an RBS which agrees with
$w$.  We can construct a spherical cancellation diagram for $w$ as
follows: let $D$ be an oriented disk with boundary circle subdivided
into $n$ vertices and arcs.  Choose a base vertex and assign the letters
of $w$ as labels to each of these vertices, starting with the base vertex 
and continuing in the order determined by the orientation.  If $u_i=[$
and $u_j=]$ is an innermost pair of matching brackets for some $1 \leq i
< j \leq n$, join the corresponding pair of vertices (that are labelled
by a pair of inverse letters) by an arc.  Deleting this pair of
brackets, find another innermost pair of matching brackets and again
add arcs to the diagram until all brackets are accounted for.  The
resulting arcs are disjoint.  The remaining vertices on boundary circle
are labelled by positive letters.  We can choose disjoint neighborhoods
around each such vertex and in each delete a small open disk.  The
resulting inner boundary circle is labelled by the corresponding inverse
letter and then joined by an arc to the corresponding outer vertex.  In
this way, a spherical cancellation diagram for $w$ is constructed.
Therefore, $w$ is quasi-positive.

\section{Acknowledgements}

The authors would like to thank Matt Hedden for sharing his interest in
quasi-positivity in braid groups.  It was these conversations which
inspired this project.  The first author would like to thank the
Department of Mathematics at the University of Michigan for support
during his stay as a Visiting Scholar from August 2017 until August
2018.  The first author also gratefully acknowledges support from the
Lyman Briggs College of Michigan State University for financial support
for travel between East Lansing and Ann Arbor, Michigan during this
period.

\bibliography{QP_free_rev3} \bibliographystyle{plain}

\end{document}